\documentclass{amsart}

\usepackage{amscd}
\usepackage{amssymb}
\usepackage{amsmath}
\usepackage{amsfonts}
\usepackage{enumerate}
\usepackage[all,cmtip]{xy}
\usepackage{tikz-cd}

\newtheorem{thm}{Theorem}[section]
\newtheorem{prop}[thm]{Proposition}
\newtheorem{lem}[thm]{Lemma}
\newtheorem{cor}[thm]{Corollary}

\theoremstyle{definition}

\theoremstyle{remark}
\newtheorem{rem}[thm]{Remark}

\newtheorem*{nota}{Notation}

\newcommand{\Zz}{\mathbb {Z}}
\newcommand{\Ff}{\mathbb {F}}
\newcommand{\Cc}{\mathbb {C}}

\newcommand{\Qq}{\mathbb {Q}}

\newcommand{\im}{\mathrm{Im\,}}

\newcommand{\Hom}{\mathrm{Hom}}
\newcommand{\xr}{\xrightarrow}

\numberwithin{equation}{section}
\usepackage[colorlinks,linktocpage]{hyperref}
\hypersetup{citecolor=blue,linkcolor=blue}
\begin{document}
	\title[A formula for the mod $p$ cohomology of $BPU(p)$]{A formula for the mod $p$ cohomology of $BPU(p)$}
	\author[F.~Fan]{Feifei Fan}
	\thanks{The author is supported by the National Natural Science Foundation of China (Grant No. 12271183) and by the GuangDong Basic and Applied Basic Research Foundation (Grant No. 2023A1515012217).}
	\address{Feifei Fan, School of Mathematical Sciences, South China Normal University, Guangzhou, 510631, China.}
	\email{fanfeifei@mail.nankai.edu.cn}

	\begin{abstract}
		We study the mod $p$ cohomology ring of the classifying space $BPU(p)$ of the projective unitary group $PU(p)$, when $p$ is an odd prime. We prove a mod $p$ formula analogous to a formula of Vistoli for the integral cohomology ring of $BPU(p)$. As an application, we give a simple topological proof of Vistoli's formula. 
	\end{abstract}
	
		\maketitle
		
	\section{Introduction}\label{sec:introduction}
	Let $BG$ be the classifying space of a topological group $G$. For a group homomorphism $\pi:G\to G'$, $B\pi:BG\to BG'$ denotes the corresponding map on classifying spaces. If $H$ is a subgroup of $G$, $i_H$ denotes the inclusion morphism $H\hookrightarrow G$. 
	
	For a compact connected Lie group $G$, let $T_G\subset G$ be a maximal torus of $G$. We denote by $W_G:=N_G(T_G)/T_G$ the Weyl group of $G$. When the group $G$ is clear from the context, we shall often omit the subscript $G$ from the notations. The image of the induced  homomorphism of cohomology rings with coefficient ring $R$
	\[Bi_T^*:H^*(BG;R)\to H^*(BT;R)\] 
    lies in $H^*(BT;R)^{W}$, the invariant subring under the conjugation $W$-action.

    By the classical work of Borel \cite{Borel53,Borel55},  $H^*(BG;\Qq)\cong H^*(BT;\Qq)^{W}$, which is a polynomial algebra generated by even dimensional elements. Similar identifications hold for cohomology with coefficients in fields of prime characteristic $p$ as
    soon as $H^*(G;\mathbb{Z})$ (or equivalently, $H^*(BG;\mathbb{Z})$) has no $p$-torsion. Except these simple cases, the cohomology of $BG$ may be quite complicated. (See for example \cite{KMS75,Tod87,Vis07,Fan24a}.)	
	
	In this paper we will focus on the mod $p$ ($p$ an odd prime) cohomology of $BG$ for the \emph{projective unitary group} $G=PU(p)=U(p)/S^1$, where $S^1\subset U(p)$ is the center. We give a formula for the mod $p$ cohomology ring of $BPU(p)$, which is analogous to a formula of Vistoli (Theorem \ref{thm:vistoli}) for the integral cohomology of $BPU(p)$. Vistoli's proof of Theorem \ref{thm:vistoli} involves many techniques of algebraic geometry, some of which were originally developed by Vezzosi \cite{Vezz00}. By contrast, our proof is topological. As an application, we will give a simple proof of Vistoli's formula.
	
	The mod $p$ (or integral) cohomology of $BPU(n)$ with $p|n$ has broad relevance in algebraic topology and algebraic geometry. It is the main tool to classify \emph{topological Azumaya algebras} over a topological space $X$, which was originally defined by Grothendieck \cite{Gro66}. It is also an important input for studying the \emph{topological period-index problem} introduced by Antieau-Williams \cite{AW14a,AW14b}, and further studied by Gu \cite{Gu19,Gu20}. Recently, Chen and Gu \cite{CG24} have found another interesting application of the cohomology of $BPU(n)$ to the \emph{topological complexity problem} in enumerative algebraic geometry. Since the cohomology of $BPU(p)$ plays a key role in the computation of the cohomology of $BPU(n)$ with $p|n$, as shown in works such as \cite{Gu21,Fan24b}, the result of this paper can be used to study the mod $p$ cohomology ring of $BPU(n)$ for general $n$.
	
	The strategy of our proof will be to determine the image of the mod $p$ cohomology ring of $BPU(p)$ under the following ring homomorphism induced by restrictions to subgroups of $PU(p)$:
	\[H^*(BPU(p);\Ff_p)\to H^*(BT_{PU(p)};\Ff_p)\times H^*(B\Gamma;\Ff_p),\]
	where $\Gamma$ is the unique (up to conjugation) maximal nontoral elementary abelian $p$-subgroup of $PU(p)$. 
	This method was also used by Kono-Yagita \cite{KY93} to determine the Steenrod algebra structure of the mod $3$ cohomology ring of $BPU(3)$, and later by Kameko-Yagita \cite{KY08} to determine the additive structure of the mod $p$ cohomology of $BPU(p)$ for all odd primes $p$. The Brown-Peterson cohomology of $BPU(p)$ was also studied in these works based on the ordinary cohomology results for $BPU(p)$. 
	
	To introduce our result, we need a few notations. Recall that the cohomology ring of $BU(n)$ with arbitrary coefficients is
	\[H^*(BU(n);R)=R[c_1,\dots,c_n],\ \deg(c_i)=2i.\]
	Here $c_i$ is the $i$th universal Chern class of the classifying space $BU(n)$ for $n$-dimensional complex bundles. It is known that 
	\begin{gather*}
		Bi_T^*:H^*(BU(n);R)\cong H^*(BT;R)^W=R[\sigma_1,\dots,\sigma_n],\\
		Bi_T^*(c_i)=\sigma_i\in H^*(BT;R)\cong R[t_1,\dots,t_n],
	\end{gather*} 
		where $\sigma_i$ is the $i$th elementary symmetric polynomial in the variables $t_1,\dots,t_n$. 
		
		Let $\nabla_n=\sum_{i=1}^n\partial/\partial t_i$ be the linear differential operator acting on $H^*(BT_{U(n)};R)$. Then it is easy to see that $\nabla_n$ preserves symmetric polynomials, so that there is a restriction action of $\nabla_n$ on the subring $H^*(BU(n);R)\subset H^*(BT_{U(n)};R)$. For elementary symmetric polynomials, one checks that 
\begin{equation}\label{eq:nabla}
	\nabla_n(\sigma_k)=(n-k+1)\sigma_{k-1}.	
\end{equation}
	Let $K_n$ be the kernel of the restriction of $\nabla_n$ to $H^*(BU(n);\Zz)$, which is naturally isomorphic to the ring of invariants $H^*(BT_{PU(n)};\Zz)^W$ (see Theorem \ref{thm:C-G}).

	Let $T^p$ be the standard maximal torus in $U(p)$ consisting of diagonal
	matrices. Consider the embedding 
	\[\Ff_p\hookrightarrow T^p,\ \ \omega\mapsto (\omega,\omega^2,\dots,\omega^{p-1},1),\ \omega=e^{2\pi i/p}.\]
	This map induces a homomorphism  
	\[\Theta_p:H^*(BT^p;\Zz)\cong \Zz[t_1,\dots,t_p]\to\Zz[\eta]/(p\eta)\cong H^*(B\Ff_p;\Zz),\ t_i\mapsto i\eta.\] 
	Let $I_p$ be the kernel of the restriction of $\Theta_p$ to $K_p\subset H^*(BT^p;\Zz)$. 
	For the symmetric polynomial $\delta=\prod_{i\neq j}(t_i-t_j)\in K_p$, one easily checks that 
	\begin{equation}\label{eq:delta}
	\Theta_p(\delta)\equiv -\eta^{p^2-p}\neq 0\mod (p\eta).
    \end{equation}
	
	Vistoli proved the following formula for the integral cohomology ring of $BPU(p)$.
	\begin{thm}[Vistoli {\cite[Theorem 3.4]{Vis07}}]\label{thm:vistoli}
		For any prime $p>2$, the integral cohomology ring of $BPU(p)$ is given by
		\[ H^*(BPU(p);\Zz)\cong\frac{K_p\otimes \Zz[x_{2p+2}]\otimes\Lambda[x_3]}{\langle px_{2p+2},\,px_3,\,I_px_{2p+2},\,I_px_3\rangle},\]
		where subscripts of generators denote degree.
	\end{thm}
	\begin{rem}
		The additive structure of $K_n$ is easily obtained from the isomorphism 
		\[K_n\otimes\Qq\cong H^*(BPU(n);\Qq)\cong H^*(BSU(n);\Qq)\cong \Qq[c_2,\dots,c_n].\]
		Furthermore, it is shown in \cite[Proposition 3.1]{Vis07} (see Corollary \ref{cor:vistoli image}) that the image of $\Theta|_{K_p}$ is generated by $\Theta(\delta)$. Hence, the $\Zz$-module structure of $H^*(BPU(p);\Zz)$ is known.
		However, the ring structure of $K_n$ is very complicated, and is unknown when $n\geq 5$ (see \cite{Vezz00} for $n=3$, and \cite{Fan24a} for $n=4$). In \cite[Theorem 1.2]{FZZZ24}, the author and collaborators proved that the first relation appears in $K_n^{12}$.
	\end{rem}

 The following theorem for the mod $p$ cohomology ring of $BPU(p)$ is the main result of this paper. Note that the additive structure of $H^*(BU(p);\Ff_p)$ was determined by Kameko and Yagita \cite[Theorem 1.7]{KY08}.
	\begin{thm}\label{thm:main}
		Let $L_p$ be the subring  of $H^*(BU(p);\Ff_p)$ generated by $K_p$ (mod $p$) and $c_1$, and let $Q_p$ be the tensor algebra $\Ff_p[x_{2p+2}]\otimes\Lambda_{\Ff_p}[x_3,x_{2p+1}]$, subscripts indicating degree, $A_p$ be the augmentation ideal of $Q_p$. Then
		\[ H^*(BPU(p);\Ff_p)\cong \frac{L_p\otimes Q_p}{\langle I_p A_p,\,c_1x_3,\,c_1x_{2p+1},\,c_1x_{2p+2}+x_3x_{2p+1}\rangle}.\]
		Here $I_p$ denotes its mod $p$ reduction.
	\end{thm}
	\begin{rem}
	$K_3$ was computed by Vezzosi in \cite{Vezz00} (see also \cite[Theorem 14.2]{Vis07}). It turns out that $K_3$ is generated by $\gamma_2=3c_2-c_1^2$, $\gamma_3=27c_3-9c_1c_2+2c_1^3$ and 
	\[\delta=4c_1^3c_3-c_1^2c_2^2-18c_1c_2c_3+4c_2^3+27c_3^2.\]
	Therefore, $L_3$ is generated by $c_1$ and $\delta\equiv c_1^3c_3-c_1^2c_2^2+c_2^3$ (mod $3$). Combining this with \eqref{eq:delta} and Theorem \ref{thm:main}, we recover the result  of Kono, Mimura and Shimada \cite[Theorem 4.11]{KMS75} for $H^*(BPU(3);\Ff_3)$.
	\end{rem}
	\begin{nota}
	For the rest of this paper, we assume that $p$ is an odd prime. We use the simplified notation $H^*(-)$ to denote the integral cohomology $H^*(-;\Zz)$, and let $\rho:H^*(-)\to H^*(-;\Ff_p)$ be the mod $p$ reduction on cohomology.
	For a fixed odd prime $p$, let $P^k$ be the $k$th Steenrod power operation and let $\beta$ be the Bockstein in the mod $p$ Steenrod algebra. 
    \end{nota}

	\section{Classical results for $BPU(n)$}
	First, we introduce some general results on the cohomology of $BPU(n)$ for arbitrary $n$. 
	Let $\pi:U(n)\to PU(n)=U(n)/S^1$ be the quotient map defining $PU(n)$. Then there is a related fibration of classifying spaces 
	\begin{equation}\label{eq:fibration}
		BU(n)\xr{B\pi} BPU(n)\xr{\chi} BBS^1\simeq K(\Zz,3).
	\end{equation}
   Let $x\in H^*(K(\Zz,3))\cong \Zz$ be the fundamental class.
   It is known that $H^k(BPU(n))=0$ for $k=1,2$, and $H^3(BPU(n))\cong\Zz/n$ is generated by $\chi^*(x)$.
   
   Let $E_*^{*,*}$ be the cohomological Serre spectral sequence with integer coefficients associated to \eqref{eq:fibration}:
   \begin{equation}\label{eq:Serre}
   	E_2^{s,t}=H^s(K(\Zz,3);H^t(BU(n)))\Longrightarrow H^{s+t}(BPU(n)).
   \end{equation}
   Since $H^*(BU(n))$ is concentrated in even dimensions, we have $E_2=E_3$.
   In \cite{Gu21}, Gu studied some higher differentials in this spectral sequence. In particular, the differential $d_3$ can be described as follows:
   \begin{prop}[Gu, {\cite[Corollary 3.10]{Gu21}}]\label{prop:d_3}
   	For $f\in E_3^{0,2k}\cong H^{2k}(BU(n))$, $\xi\in E_3^{i,0}\cong H^i(K(\Zz,3))$, and $f\xi\in E_3^{i,2k}$, we have
   	\[
   	d_3(f\xi)=\nabla_n(f)x\xi\in E_3^{i+3,2k-2}.
   	\]
   \end{prop}
   In fact, Proposition \ref{prop:d_3} is also true for the corresponding spectral sequence with arbitrary coefficients, and this is easily seen from its proof in \cite{Gu21}.
	
	Here is an important result of Crowley and Gu about the images of the restriction maps $B\pi^*$ and $Bi^*_T$ on $H^*(BPU(n))$.
	\begin{thm}[Crowley-Gu, {\cite[Theorem 1.3]{CG21}}]\label{thm:C-G}
		For $H^*(BPU(n))$ and the spectral sequence \eqref{eq:Serre}, there are isomorphisms
		\[K_n=E_4^{0,*}=E_\infty^{0,*}=\im B\pi^*\cong \im Bi_T^*=H^*(BT)^W\cong H^*(BPU(n))/torsion.\]
	\end{thm}
	
	The next result due to Vezzosi \cite[Corollary 2.4]{Vezz00} is about the ideal of torsion elements of $H^*(BPU(n))$. Vezzosi's original statement and proof are developed for Chow rings. Here we give a similar topological proof for cohomology.
	\begin{thm}\label{thm:Vezzosi}
		$H^*(BPU(n))$ has only $n$-torsion.
	\end{thm}
	\begin{proof}
		The fiber bundle $U(n-1)\to U(n)\to S^{2n-1}$ induces a fiber bundle 
		\[U(n-1)\to PU(n)\to \Cc P^{n-1},\]
		which induces another fiber bundle 
		\[\Cc P^{n-1}\to BU(n-1)\xr{\nu} BPU(n).\]
		Let $\tau: H^*(BU(n-1))\to H^*(BPU(n))$ be the transfer map defined in \cite{BG75}. Then $\tau\nu^*=$ multiplication by $n=\chi(\Cc P^{n-1})$ by \cite[Theorem 5.5]{BG75}. 
		Since $H^*(BU(n-1))$ is torsion free, $\nu^*$ sends any torsion element $\alpha\in H^*(BPU(n))$ to zero, hence $\tau\nu^*(\alpha)=n\alpha=0$.
	\end{proof}
	
	Now we focus on $PU(p)$ and its elementary abelian $p$-subgroups. Recall that for an elementary abelian $p$-subgroup $E$ of a Lie group $G$, the \emph{Weyl group $W_G(E)$} is defined to be the quotient group $N_G(E)/C_G(E)$ of the normalizer $N_G(E)$ of $E$ by its centralizer $C_G(E)$. Clearly, the map $Bi_E^*:H^*(BG;\Ff_p)\to H^*(BE;\Ff_p)$ has image in the ring of invariants $H^*(BE;\Ff_p)^{W_G(E)}$.

	Let $\omega=e^{2\pi i/p}$, and let $\tilde\sigma,\tilde\tau\in U(p)$ be
	\[\tilde\sigma=\begin{pmatrix}
		0& 1\\
		I_{p-1} &0
	\end{pmatrix},\quad
	\tilde\tau=(\omega,\omega^2,\dots,\omega^{p-1},1)\in T^p.\]
	Let $\tilde\Gamma$ be the subgroup of $U(p)$ generated by $\tilde\sigma$ and $\tilde\tau$, and let $\sigma$, $\tau$, $\Gamma$ denote the corresponding images in the quotient group $PU(p)$. One easily checks that $\tilde\tau\tilde\sigma=\omega\tilde\sigma\tilde\tau$, so $\sigma$ and $\tau$ commute in $PU(p)$, i.e. $\Gamma\cong \Ff^2_p$. 
	Thus 
	\[H^*(B\Gamma;\Ff_p)\cong \Ff_p[\xi,\eta]\otimes\Lambda_{\Ff_p}[a,b],\] 
	where $a,b\in H^1(B\Gamma;\Ff_p)\cong \Hom(\Gamma,\Ff_p)$ are defined by $a(\sigma)=b(\tau)=1$, $a(\tau)=b(\sigma)=0$, and $\xi=\beta(a)$, $\eta=\beta(b)$. The integral cohomology of $B\Gamma$ is 
	\[H^*(B\Gamma)\cong\Zz[\xi,\eta,s]/(p\xi,p\eta,ps,s^2)\]
	with $\rho(\xi)=\xi$, $\rho(\eta)=\eta$, $\rho(s)=\xi b-\eta a$.
	
	By a result of Andersen et al. \cite[Theorem 8.5]{AGMV08}, we have $W_{PU(p)}(\Gamma)=SL_2(\Ff_p)$. Let $y=ab\in H^*(B\Gamma;\Ff_p)$.
	Since $SL_2(\Ff_p)$ acts naturally on $\Ff_p\{a,b\}$, it is easy to check that $y\in H^*(B\Gamma;\Ff_p)^{SL_2(\Ff_p)}$. Since the Bockstein homomorphism and Steenrod powers are $SL_2(\Ff_p)$-equivariant, $s:=\beta(y)=\xi b-\eta a$, $z:=P^1(s)=\xi^pb-\eta^pa$ and $f:=\beta(z)=\xi^p\eta-\eta^p\xi$ lie in $H^*(B\Gamma;\Ff_p)^{SL_2(\Ff_p)}$. Moreover, let 
	\[h:=\xi^{p^2-p}+\eta^{p-1}(\xi^{p-1}-\eta^{p-1})^{p-1}.\]
	Then by the result of Dickson  \cite{Dickson11} for the invariants of the action of $SL_n(\Ff_p)$ on $\Ff_p[n]$ (here we use it for the special case $n=2$), we have
	\begin{equation}\label{eq:Dickson}
		\Ff_p[\xi,\eta]^{SL_2(\Ff_p)}=\Ff_p[f,h]\subset H^*(B\Gamma;\Ff_p)^{SL_2(\Ff_p)}.
	\end{equation}

	The following proposition is the special case $n=2$ of M\`ui's theorem \cite{Mui75} (see also \cite[Theorem 1.1]{KM07}) for the ring of $SL_n(\Ff_p)$-invariants of the tensor product algebra $\Ff_p[n]\otimes\Lambda_{\Ff_p}[n]$. 
	\begin{prop}\label{prop:mod p invariants}
		$H^*(B\Gamma;\Ff_p)^{SL_2(\Ff_p)}$ is the subring of $H^*(B\Gamma;\Ff_p)$ generated by $f,h,s,y,z$, which is isomorphic to the following quotient ring
		\[\frac{\Ff_p[f,h]\otimes\Lambda_{\Ff_p}[s,y,z]}{(ys,\,yz,\,fy+sz)}.\]
	\end{prop}
	\begin{cor}[Vistoli, {\cite[Proposition 5.10]{Vis07}}]\label{cor:vistoli integral}
		$H^*(B\Gamma)^{SL_2(\Ff_p)}$ is the subring of $H^*(B\Gamma)$ generated by $s,f,h$.
	\end{cor}
	\begin{cor}[Vistoli, {\cite[Proposition 3.1]{Vis07}}]\label{cor:vistoli image}
		$K_p$ is generated by $I_p$ and $\delta$. In other words, $\Theta_p(K_p)=\langle\Theta_p(\delta)\rangle=\langle\eta^{p^2-p}\rangle$.
	\end{cor}
	\begin{proof}
	Since the image of $s,f$ in $H^*(B\Ff_p\{\tau\})$ is zero and the image of $h$ is $\eta^{p^2-p}$, the image of $H^*(BPU(p))$ lies in $\langle\eta^{p^2-p}\rangle$. Consider the commutative diagram:
	\[\xymatrix{
		H^*(BPU(p))\ar[r]\ar[d]^{B\pi^*}&H^*(B\Ff_p\{\tau\})\ar@{=}[d]\\
		H^*(BU(p))\ar[r]^{\Theta_p}&H^*(B\Ff_p\{\tau\})}
	\]
	Since $\delta\in K_p=\im B\pi^*$ by Theorem \ref{thm:C-G} and $\Theta_p(\delta)=-\eta^{p^2-p}$, the desired equation follows immediately.
	\end{proof}

	The following result of Vavpeti\v c and Viruel plays a central role in the proof of Theorem \ref{thm:main}.
	\begin{thm}[Vavpeti\v c-Viruel, {\cite[Theorem 2.5]{VV05}}]\label{thm:V-V}
		The restriction map 
		\[H^*(BPU(p);\Ff_p)\to H^*(BT_{PU(p)};\Ff_p)\times H^*(B\Gamma;\Ff_p)\]
		is injective.
	\end{thm} 
	\begin{cor}\label{cor:V-V}
		The map $B\pi^*\times Bi_\Gamma^*:$ 
		\[H^*(BPU(p);\Ff_p)\to H^*(BU(p);\Ff_p)\times H^*(B\Gamma;\Ff_p)\]
		is injective.
	\end{cor}
	\begin{proof}
		The natural map $H^*(BT_{PU(p)};\Ff_p)\to H^*(BT_{U(p)};\Ff_p)$ is injective, so by Theorem \ref{thm:V-V} the map
		\[H^*(BPU(p);\Ff_p)\to H^*(BT_{U(p)};\Ff_p)\times H^*(B\Gamma;\Ff_p)\] 
		is injective, and this map factors through 
		the map in the corollary.
	\end{proof}
	
	\begin{cor}\label{cor:integral injection}
		The  map $B\pi^*\times Bi_\Gamma^*:$
		\[H^*(BPU(p))\to H^*(BU(p))\times H^*(B\Gamma)\]
		is injective.
	\end{cor}
	\begin{proof}
		Consider the commutative diagram
		\[\xymatrix{
			H^*(BPU(p))\ar[r]\ar[d]^{\rho}
			&H^*(BU(p))\times H^*(B\Gamma)\ar[d]^{\rho}\\
			H^*(BPU(p);\Ff_p)\ar[r]
			&H^*(BU(p);\Ff_p)\times H^*(B\Gamma;\Ff_p)}
		\]
		The kernel of the left vertical map is $pH^*(BPU(p))$. But $pH^*(BPU(p))$ is a subset of $H^*(BPU(p))/torsion$ by Theorem \ref{thm:Vezzosi}, and so the restriction of $B\pi^*$ to $pH^*(BPU(p))$ is injective, then the result follows by Corollary \ref{cor:V-V}.
	\end{proof}

	\section{Proof of Theorem \ref{thm:main}}\label{sec:proof of main thm}
	Before giving the proof of Theorem \ref{thm:main}, we first establish a few preliminary facts.
	\begin{lem}\label{lem:Bockstein}
		Let $u\in H^*(BPU(p);\Ff_p)$ be a nonzero element such that $i_\Gamma^*(u)=0$. Then $u$ is a permanent cycle in the Bockstein spectral sequence associated to $H^*(BPU(p);\Ff_p)$.
	\end{lem}
	\begin{proof}
		Since $Bi_\Gamma^*(u)=0$, by Theorem \ref{thm:V-V} $Bi^*_T(u)\neq 0$, so that $u$ is even dimensional. If $\alpha$ is not a permanent cycle in the Bockstein spectral sequence of $H^*(BPU(p);\Ff_p)$, then there exists $v\in H^*(BPU(p);\Ff_p)$ such that either $\beta(v)=u$ or $\beta(u)=v$, since $H^*(BPU(p))$ has only $p$-torsion by Theorem \ref{thm:Vezzosi}. The first case is impossible because $Bi^*_T(u)$ is a permanent cycle in the Bockstein spectral sequence of $H^*(BT;\Ff_p)$. For the second case, $v$ is odd dimensional so that $Bi^*_T(v)=0$, and then $Bi_\Gamma^*(v)\neq 0$ by Theorem \ref{thm:V-V}.
		But $Bi^*_\Gamma(v)=Bi^*_\Gamma\beta(u)=\beta Bi^*_\Gamma(u)=0$, a contradiction.
	\end{proof}
	
	\begin{lem}\label{lem:image i_gamma}
	The image of the map $Bi_\Gamma^*:H^*(BPU(p);\Ff_p)\to H^*(B\Gamma;\Ff_p)$ is $H^*(B\Gamma;\Ff_p)^{SL_2(\Ff_p)}$.
	\end{lem}
	\begin{proof}
		In the notation of Proposition \ref{prop:mod p invariants}, we need to show that $f,h,s,y,z\in\im Bi_\Gamma^*$. A brief inspection of the quotient ring in Proposition \ref{prop:mod p invariants} shows that $H^*(B\Gamma;\Ff_p)^{SL_2(\Ff_p)}$ in dimensions $2,3,2p+1,2p+2,2p^2-2p$ are $\Ff_p$, which are spanned by $y,s,z,f,h$ respectively. Since $H^3(BPU(p))=\Ff_p\{\chi^*(x)\}$, there exist $x_i\in H^i(BPU(p);\Ff_p)$, $i=2,3$, such that $\beta(x_2)=x_3=\rho\chi^*(x)$. Clearly, $Bi^*_T(x_3)=0$, so rescaling $s$ if necessary, $Bi^*_\Gamma(x_3)=s$ by Theorem \ref{thm:V-V}, and therefore $Bi^*_\Gamma(x_2)=y$. Note that $z=P^1(s)$ and $f=\beta(z)$. Hence $s,y,z,f\in\im Bi^*_\Gamma$. 
		
		To see that $h\in\im Bi^*_\Gamma$, we use the commutative diagram
		\[\xymatrix{
			H^*(BPU(p))\ar[r]\ar[d]^{B\pi^*}&H^*(B\Ff_p\{\tau\})\ar@{=}[d]\\
			H^*(BU(p))\ar[r]^{\Theta_p}&H^*(B\Ff_p\{\tau\})}
		\]
Since $\delta\in K_p=\im B\pi^*$ by Theorem \ref{thm:C-G} and $\Theta_p(\delta)\neq 0$ by \eqref{eq:delta}, there exists $x_{2p^2-2p}\in H^{2p^2-2p}(BPU(p))$ such that $B\pi^*(x_{2p^2-2p})=\delta$, and therefore its image in $H^{2p^2-2p}(B\Ff_p\{\tau\})$ is nonzero.
Since the upper map factors through \[H^{2p^2-2p}(BPU(p))\xr{Bi_\Gamma^*}H^{2p^2-2p}(B\Gamma)^{SL_2(\Ff_p)}=\Ff_p\{h\},\]it follows that $Bi_\Gamma^*(x_{2p^2-2p})=-h$, noting that the images of $\delta$ and $h$ in $H^*(B\Ff_p\{\tau\})$ are $-\eta^{p^2-p}$ and $\eta^{p^2-p}$ respectively.
 \end{proof}
	
	\begin{lem}\label{lem:image pi}
		The image of the map $B\pi^*:H^*(BPU(p);\Ff_p)\to H^*(BU(p);\Ff_p)$ is $L_p$. Here $L_p$ is the subring as in Theorem \ref{thm:main}. 
	\end{lem}
	\begin{proof}
	Let ${_p}E_*^{*,*}$ be the Serre spectral sequence associated to \eqref{eq:fibration} with coefficients in $\Ff_p$. Applying Proposition \ref{prop:d_3} and \eqref{eq:nabla} to ${_p}E_*^{*,*}$, we see that $c_1\in {_p}E_4^{0,2}={_p}E_\infty^{0,2}$ (the equality comes from degree reasons), and so $c_1\in \im B\pi^*={_p}E_\infty^{0,*}$. Hence $L_p\subset\im B\pi^*$, using Theorem \ref{thm:C-G}.
	
	It remains to prove the opposite inclusion. 
	Let $x_2,x_3,x_{2p^2-2p}$ be as in the proof of Lemma \ref{lem:image i_gamma}, and let $x_{2p+1}=P^1(x_3)$, $x_{2p+2}=\beta(x_{2p+1})$. Then we have the following equations for mod $p$ cohomology:
	\begin{gather}\label{eq:Bpi}
		\begin{gathered}
		B\pi^*(x_2)=c_1,\ \ B\pi^*(x_{2p^2-2p})=\delta,\\ B\pi^*(x_3)=B\pi^*(x_{2p+1})=B\pi^*(x_{2p+2})=0,
		\end{gathered}
	\end{gather}
	\begin{gather}\label{eq:Bgamma}
	\begin{gathered}
		Bi^*_\Gamma(x_2)=y,\ \ Bi^*_\Gamma(x_3)=s,\ \ Bi^*_\Gamma(x_{2p+1})=z,\\ Bi^*_\Gamma(x_{2p+2})=f,\ \ Bi^*_\Gamma(x_{2p^2-2p})=-h.
	\end{gathered}
    \end{gather}
    (Remark: from the above discussion we only have $B\pi^*(x_2)=kc_1$ for some nonzero constant $k\in\Ff_p$. But $k$ is actually $1$, which  comes from the fact that $d_3(c_1)=px$ in $E_*^{*,*}$ by Proposition \ref{prop:d_3}, and the fact that $\beta(x_2)=x_3=\chi^*(x)$.)
	
  Let $I$ be the ideal of $H^*(BPU(n);\Ff_p)$ generated by the above $x_i$'s. Then
  $B\pi^*(I)\subset \Ff_p[c_1,\delta]\subset L_p$ by \eqref{eq:Bpi}. On the other hand, $Bi_\Gamma^*(I)=H^{>0}(B\Gamma;\Ff_p)^{SL_2(\Ff_p)}$ by Proposition \ref{prop:mod p invariants} and \eqref{eq:Bgamma}. Hence, for any $u\in H^{>0}(BPU(p);\Ff_p)$, there exists $v\in I$ such that $Bi^*_\Gamma(u+v)=0$, and then $B\pi^*(u+v)\in \rho(K_p)\subset L_p$ by Lemma \ref{lem:Bockstein} and Theorem \ref{thm:C-G}. It follows that $B\pi^*(u)\in L_p$ since $B\pi^*(v)\in L_p$. 
  \end{proof}

  \begin{lem}\label{lem:kernel J_p}
  	Let $J_p$ be the kernel of the map $Bi_\Gamma^*$ on $H^*(BPU(p);\Ff_p)$. Then $B\pi^*(J_p)=\rho(I_p)$ in $H^*(BU(p);\Ff_p)$.
  \end{lem}
  \begin{proof}
  	By Lemma \ref{lem:Bockstein}, $J_p\subset \rho(H^*(BPU(p)))$, hence 
   $B\pi^*(J_p)\subset\rho(K_p)$ by Theorem \ref{thm:C-G}. It is clear that $J_p$ maps to zero in $H^*(B\Ff_p\{\tau\};\Ff_p)$, so we have the inclusion $B\pi^*(J_p)\subset \rho(I_p)$. 
   
   For the opposite inclusion, let $v$ be an arbitrary element of $\rho(I_p)$. Since $\rho(I_p)\subset \rho(K_p)=\rho B\pi^*(H^*(BPU(p))$, there exists $u\in \rho(H^*(BPU(p)))$ such that $ B\pi^*(u)=v$. By Corollary \ref{cor:vistoli integral}, $Bi_\Gamma^*(u)\in \langle s,f,h\rangle$. But $u$ maps to zero in $H^*(B\Ff_p\{\tau\};\Ff_p)$, so $Bi_\Gamma^*(u)$ actually lies in $\langle s,f\rangle$. Let $x_3$, $x_{2p+2}$ be the elements defined in the proof of Lemma \ref{lem:image pi}. Then by \eqref{eq:Bgamma}, there exists $u'\in\langle x_3,x_{2p+2}\rangle$ such that $Bi_\Gamma^*(u+u')=0$, i.e., $u+u'\in J_p$. But since $B\pi^*(x_3)=B\pi^*(x_{2p+2})=0$ by \eqref{eq:Bpi}, we have $B\pi^*(u+u')=v$, and the inclusion $\rho(I_p)\subset B\pi^*(J_p)$ follows.
  \end{proof}
  
  \begin{prop}\label{prop:image}
  	The image of $H^*(BPU(p);\Ff_p)$ in $H^*(BU(p);\Ff_p)\times H^*(B\Gamma;\Ff_p)$ is the subring $R$ generated by $(\rho(I_p),0)$ and the following elements:
  	\[(0,s),\ (0,z),\ (0,f),\ (c_1,y),\ (\delta,-h).\]
  \end{prop}
  \begin{proof}
  	By \eqref{eq:Bpi}, \eqref{eq:Bgamma} and Lemma \ref{lem:kernel J_p}, $R\subset \im(B\pi^*\times Bi_\Gamma^*)$. On the other hand, let $R'\subset R$ be the subring generated by $(0,s)$, $(0,z)$, $(0,f)$, $(c_1,y)$, $(\delta,-h)$. Then for any $(u,v)\in \im(B\pi^*\times Bi_\Gamma^*)$, there exists $(u',v)\in R'$ by Proposition \ref{prop:mod p invariants}, so that $u-u'\in B\pi^*(J_p)$. Hence by Lemma \ref{lem:kernel J_p}, $u'-u\in \rho(I_p)$, i.e., $(u-u',0)\in R$. This implies that $(u,v)\in R$, and the  opposite inclusion $ \im(B\pi^*\times Bi_\Gamma^*)\subset R$ follows.
  \end{proof}
  By Corollary \ref{cor:V-V} and Proposition \ref{prop:image}, we will prove Theorem \ref{thm:main} by constructing an isomorphism between $R=\im(B\pi^*\times Bi_\Gamma^*)$ and the quotient ring on the right side of the isomorphism in Theorem \ref{thm:main}.
  \begin{proof}[Proof of Theorem \ref{thm:main}]
	Let $L_p,Q_p,A_p$ be as in Theorem \ref{thm:main}. Define a homomorphism $\Phi:L_p\otimes Q_p\to R$ by 
	\begin{gather*}
		x_{3}\mapsto (0,s),\ x_{2p+1}\mapsto(0,z),\ x_{2p+2}\mapsto(0,f),\\
		c_1\mapsto(c_1,y),\ \delta\mapsto(\delta,-h),\ u\mapsto(u,0),\ u\in \rho(I_p).
	\end{gather*}
	Clearly, $\Phi$ is a well-defined surjection by Proposition \ref{prop:image}.
	Let 
	\[I=\langle I_p A_p,\,c_1x_3,\,c_1x_{2p+1},\,c_1x_{2p+2}+x_3x_{2p+1}\rangle\]
	 be the ideal of relations in Theorem \ref{thm:main}. We need to show that $\ker\Phi=I$. The inclusion $I\subset \ker\Phi$ is obvious. Suppose that $\Phi(e)=0$ for an element $e\in L_p\otimes Q_p$. Since the projection of $R$ to the second factor has image $H^*(B\Gamma;\Ff_p)^{SL_2(\Ff_p)}$ by Lemma \ref{lem:image i_gamma}, it follows from Proposition \ref{prop:mod p invariants} that 
	 \[e\in I':=\langle I_p,\,c_1x_3,\,c_1x_{2p+1},\,c_1x_{2p+2}+x_3x_{2p+1}\rangle.\]
	 Moreover, by Lemma \ref{lem:image pi} the projection of $R$ to the first factor has image $L_p$, hence $e\in L_p\otimes A_p$, then $e\in I=I'\cap(L_p\otimes A_p)$.
  \end{proof}
	
		\section{A simple proof of Theorem \ref{thm:vistoli}}\label{sec:proof of vistoli thm}
		The argument is very similar to the proof of Theorem \ref{thm:main}. First we have an analog of Proposition \ref{prop:image}.
	\begin{prop}\label{prop:integral image}
		The image of $H^*(BPU(p))$ in $H^*(BU(p))\times H^*(B\Gamma)$ is the subring $R_0$ generated by $(I_p,0)$ and the elements $(0,s)$, $(0,f)$, $(\delta,-h)$.	
		\end{prop}
		\begin{proof}
			Let $B\pi^*$ and $Bi_\Gamma^*$ be the restriction map on $H^*(BPU(p))$, and let $J_p=\ker Bi_\Gamma^*$. Similar to Lemma \ref{lem:kernel J_p}, one can show that $B\pi^*(J_p)=I_p$. Hence $R_0\subset \im(B\pi^*\times Bi_\Gamma^*)$, since $(0,s)$, $(0,f)$, $(\delta,-h)\in \im(B\pi^*\times Bi_\Gamma^*)$ as we have seen.
			Using Corollary \ref{cor:vistoli integral} and arguing as in the proof of Proposition \ref{prop:image}, we get the opposite inclusion.  
		\end{proof}
		
		Finally, one can imitate the proof of Theorem \ref{thm:main} to finish the proof of Theorem \ref{thm:vistoli}, using Corollary \ref{cor:integral injection} and constructing an isomorphism between $R_0$ and the quotient ring on the right side of the isomorphism in Theorem \ref{thm:vistoli}.
		
		\section*{Acknowledgment}
		The author thanks the anonymous referee for constructive and helpful suggestions.

	\bibliography{M-A}
	\bibliographystyle{amsplain}
\end{document}